\documentclass[preprint,12pt]{elsarticle}

\usepackage{lipsum}
\usepackage{amsfonts}
\usepackage{graphicx}
\usepackage{epstopdf}
\usepackage{algorithmic}
\ifpdf
\DeclareGraphicsExtensions{.eps,.pdf,.png,.jpg}
\else
\DeclareGraphicsExtensions{.eps}
\fi

\usepackage{comment}

\usepackage{graphicx}
\usepackage{color}        
\usepackage{mdframed}
\usepackage{subfigure}

\usepackage{amsmath,amssymb,mathrsfs}
\usepackage{mathtools,xparse}

\newcommand{\pushright}[1]{\ifmeasuring@#1\else\omit\hfill
   $\displaystyle#1$\fi\ignorespaces}
\makeatother

\DeclarePairedDelimiter{\norm}{\lVert}{\rVert}
\newcommand{\rfb}[1]{\mbox{\rm
		(\ref{#1})}\ifx\undefined\stillediting\else:\fbox{$#1$}\fi} 
\newcommand{\rarrow} {\mathop{\rightarrow}}                
\newcommand{\FORALL} {{\hbox{$\hskip 11mm \forall \;$}}}
\newcommand{\e}      {{\varepsilon}}
\newcommand{\m}      {{\hbox{\hskip 1pt}}}
\newcommand{\mm}     {{\hbox{\hskip 0.5pt}}}

\renewcommand{\l}    {{\lambda}}
\renewcommand{\L}    {{\Lambda}}
\renewcommand{\Re}   {{\rm Re\,}}
\newcommand{\nline}  {{\mathbb N}}
\newcommand{\rline}  {{\mathbb R}}
\newcommand{\R}      {{\mathbb R}}
\newcommand{\N}      {{\mathbb N}}
\newcommand{\tline}  {{\mathbb T}}

\newcommand{\dd}     {{\rm d}\hbox{\hskip 0.5pt}}
\newcommand{\sbluff} {{\hbox{\raise 9pt \hbox{\mm}}}}
\newcommand{\Dscr}   {{\cal D}}
\newcommand{\Lscr}   {{\cal L}}

\newcommand{\half}   {{\frac{1}{2}}}

\newcommand{\ed}[1]{{\color{blue} #1}}

\numberwithin{equation}{section}

\newtheorem{theorem}{Theorem}[section]
\newtheorem{remark}[theorem]{Remark}
\newtheorem{proposition}[theorem]{Proposition}
\newtheorem{lemma}[theorem]{Lemma}
\newtheorem{definition}[theorem]{Definition}
\newtheorem{corollary}[theorem]{Corollary}

\journal{Systems \& Control Letters}

\begin{document}

\begin{frontmatter}



\title{The strong stability of the Perron-Frobenius semigroup 
       and almost global attractivity}


\author{Pietro Lorenzetti and George Weiss}

\address{School of Electrical Engineering, Tel Aviv University, 
Ramat Aviv 69978, Israel,
emails: plorenzetti@tauex.tau.ac.il, gweiss@tauex.tau.ac.il}

\begin{abstract}
We discuss some useful properties of the solution map (flow) of a
nonlinear dynamical system with a finite-dimensional state space.
Then, we introduce the Perron-Frobenius semigroup, and we prove that
it is a positive strongly continuous semigroup of contractions. We
show that, given a nonlinear system and an invariant set, this set is
an almost global attractor if and only if certain Perron-Frobenius
semigroups associated to the nonlinear system are strongly stable.
Unlike other works on the Perron-Frobenius semigroup from the
literature, we do not require the existence of a compact and invariant
state-space for the dynamical system, we allow trajectories with
finite escape time, and we do not require the attractor to be locally
(Lyapunov) stable. Two simple examples are used throughout the paper
to illustrate the theory.
\end{abstract}

\begin{keyword} almost global asymptotic stability, almost global 
attractor, Perron-Frobenius semigroup, density function.
\end{keyword}

\end{frontmatter}

\section{Introduction} \label{sec:1} 

Many nonlinear systems (e.g., pendulum-like systems
\cite{Reitmann1992} or power systems \cite{Bullo2019}) unavoidably
possess unstable equilibrium points. For such systems, global
asymptotic stability is not possible and the best we can hope for is
that the system has an almost global attractor \cite{Angeli2004,
Rantzer2001,Rajaram2010}. We say that an invariant set $\L$ is an
\textit{almost global attractor} if, for almost every initial
condition (in the sense of Lebesgue measure), the corresponding state
trajectory converges to $\L$. The set $\L$ may or may not contain
equilibrium points.

A simple example of a system with an almost global attractor is a
rotating pendulum with viscous friction \cite{Reitmann1992}. For
instance, a possible choice is the set of all of the (infinitely many)
stable equilibrium points. A more involved example is a grid-connected
synchronous generator (with constant excitation current), which has
been proved to possess an almost global attractor, see, e.g.,
\cite{Bara2017,Lorenzetti2022TPS,Natarajan2018}. Further examples of
such systems are phase-locked loops, as proved in \cite{Angeli2004,
Angeli2003,LorRei24, Rantzer2001CDC,Zonetti2022}. We believe that
other meaningful examples of systems with almost global attractors may
arise in the context of power systems \cite{Lorenzetti2023}. As
discussed in \cite{Bullo2019,Dorfler2012,art153,Weiss2019}, a power
system can be modeled as a system of coupled oscillators with multiple
stable and unstable equilibrium points. In particular, in
\cite{Dorfler2012,Weiss2019} the region of attraction of a preferred
equilibrium point for the network reduced power system (NRPS) model of
a power system is investigated. However, it is still unclear under
which conditions (if any) does the NRPS model possess an almost global
attractor.

We investigate the relationship between the stability properties of
the Perron-Frobenius semigroup and the existence of an almost global
attractor for nonlinear continuous-time systems. Our contribution is
twofold. First, under mild regularity conditions, we prove that the
Perron-Frobenius semigroup is a positive strongly continuous semigroup
of contractions\footnote{We adopt here the terminology used in
Functional Analysis (see, e.g., \cite{Brezis2011}), i.e., given a
mapping $T:X\to Y$ such that $\|T(x_1)-T(x_2)\|_Y\leq\rho\|x_1-x_2
\|_X$, we say that $T$ is a \textit{contraction} if $0<\rho\leq 1$.
According to the terminology of Variational Analysis (see, e.g.,
\cite{Rockafellar2009}), this property is referred to as
\textit{non-expansiveness}.}.  Second, we prove the following: Given a
nonlinear system and an invariant set $\L\subset\rline^n$, the strong
stability of the family of Perron-Frobenius semigroups (defined
outside certain invariant sets containing $\Lambda$) is equivalent to
$\Lambda$ being an almost global attractor. Related results can be
found in \cite{Vaidya2008} for discrete-time systems, and in
\cite{Mauroy2013}, where it is shown that the stability of the Koopman
operator is equivalent to the global asymptotic stability of an
attractor for continuous-time systems. Besides, the relationship
between the Perron-Frobenius semigroup and the almost global
attractivity property of a set is studied in \cite{Rajaram2015,
Rajaram2013,Rajaram2010}, under the following assumptions: the flow of
the system is nonsingular in the sense of 
\cite[Definition~3.2.2]{Lasota1998}, there exists a compact forward
invariant set for the system trajectories, and the attractor is
locally stable in the sense of Lyapunov (almost everywhere locally
stable in \cite{Vaidya2008,Vaidya2010}). Similar assumptions are in
\cite{Mauroy2013}. The work \cite{Karabacak2018} proposes to relax
some of these strong requirements by removing the assumption on the
local stability of the attractor, and by replacing the existence of a
compact forward invariant set with the existence of trajectories for
all time and for all initial conditions. Here, none of the above is
assumed: Trajectories are allowed to exist locally, there is no
compact forward invariant set for the flow of the system, and the
attractor need not be locally stable.

The paper is organized as follows. In Section~\ref{sec:2}, we
introduce some preliminaries on measure theory, which lay the
foundation of our work. In Section~\ref{sec:prob}, we discuss some
useful properties of the solution map (flow) of a dynamical system. In
Section~\ref{sec:aGAS}, we define the property of almost global
attractivity of a set, which is the object of our studies. In
Section~\ref{sec:3}, we define the Perron-Frobenius semigroup and we
prove that it is a positive strongly continuous semigroup of
contractions. Finally, we present our main result in
Section~\ref{sec:4}, and we provide some concluding remarks in
Section~\ref{sec:5}.

\section{Preliminaries on Measure Theory} \label{sec:2} 

We assume that the reader is familiar with the basic concepts from
measure theory, such as $\sigma$-algebra, (non-negative) measure,
measure space, $\sigma$-finite measure space, measurable function,
Borel and Lebesgue measurable sets in $\R^n$, Lebesgue measure. Good
introductions on measure theory can be found in
\cite{Cohn2013,Lasota1998,Rudin1987}.  We recall here some facts that
will be needed.

\begin{definition}
Let $(X,\mathcal{A},m)$ be a measure space. We define
$\mathscr{L}^1(X,\mathcal{A},m)$ as the vector space of all measurable
functions $f:X\to \R$ satisfying
\begin{equation}\label{eq:L_1}
   \|f\|_1 \m=\m \int_X|f(x)| m(\dd x)<\infty.
\end{equation}
\end{definition}

We identify two elements $f,g\in\mathscr{L}^1(X,A,m)$ if
$\|f-g\|_1=0$, or, equivalently, if the set of points $x\in X$ where
$f(x)\neq g(x)$ has measure zero. The space of equivalence classes
obtained in this way is denoted by $L^1(X,A,m)$. With the norm
inherited from \rfb{eq:L_1}, $L^1(X,A,m)$ is a Banach space.

\begin{definition}\label{def:abs_cont}
Let $(X,\mathcal{A},m)$ be a measure space. The measure $\nu$ on 
$\mathcal{A}$ is said to be \textit{absolutely continuous} with 
respect to $m$ if, for every $A\in\mathcal{A}$, \vspace{-1.5mm}
$$ m(A) = 0 \quad \Longrightarrow \quad \nu(A) = 0.$$
\end{definition}

Now we state the Radon-Nikodym theorem, which is a fundamental result
in measure theory. The proof can be found, e.g., in
\cite[Section~4.2]{Cohn2013}.

\begin{theorem} \label{thm:Radon}
Let $(X,\mathcal{A},m)$ be a $\sigma$-finite measure space. If the
measure $\nu$ is absolutely continuous with respect to $m$, then there
exists a unique $f_\nu\in L^1(X,\mathcal{A},m)$, with $f_\nu\geq 0$,
such that \vspace{-1.5mm}
$$\nu(A) \m=\m \int_Af_\nu(x)m(\dd x) \FORALL A\in\mathcal{A}.$$
\end{theorem}

The function $f_\nu$ appearing in the above theorem is called the
\textit{Radon-Nikodym derivative} of $\nu$ with respect to $m$.

\section{The Solution Map (Flow) of a Dynamical System}
\label{sec:prob} 

We work in the measure space $(\R^n,\Lscr(\R^n),\mu)$, where $\Lscr
(\R^n)$ denotes the Lebesgue $\sigma$-algebra in $\R^n$ and $\mu$ 
denotes the Lebesgue measure, see \cite[Section~1.3]{Cohn2013}. Our 
results can be extended to any measure $\nu$, which is absolutely 
continuous with respect to $\mu$. Consider the dynamical system 
described by \vspace{-1.5mm}
\begin{equation} \label{eq:ODE}
   \dot{x} \m=\m f(x),
\end{equation}
with $f\in C^1(\R^n;\R^n)$. We denote by $\phi_t(x_0)$ the
\textit{flow} (solution map) of the system \rfb{eq:ODE} from an
initial condition $x_0\in\R^n$, see \cite[Section~2.5]{Perko1998}.
Since $f\in C^1(\R^n;\R^n)$, for each fixed $x_0\in\R^n$ there exists
a maximal open interval $I_{x_0}\subset\R$ such that $0\in I_{x_0}$
and the system \rfb{eq:ODE} has a unique solution $x(t)=\phi_t(x_0)$
defined on $I_{x_0}$, with $x(0)=\phi_0(x_0)=x_0$, see \cite[Th.~1,
Sec.~2.4]{Perko1998}. For each fixed $t\in\R$, the flow $\phi_t$ is
defined on an open subset $D_t\subset\R^n$, $D_t=\{x_0\in\R^n 
\ | \ t\in I_{x_0}\}$, and $\phi_t\in C^1(D_t;
\R^n)$, see \cite[Th.~1, Sec.~2.5]{Perko1998}. (If $f$ is only locally
Lipschitz, then also $\phi_t$ is locally Lipschitz, see
\cite[Theorem~4.34]{Logemann2014}.) The flow has the \textit{group
property}, i.e., $\phi_t(\phi_s(x_0))=\phi_{s+t}(x_0)$ whenever
$s,s+t\in I_{x_0}$. For every $A\subset\R^n$ and for all $t\in\R$,
we denote $\phi^{-1}_t(A)=\{x\in\R^n \ | \ \phi_t(x)\in A\}$. The
system \rfb{eq:ODE} is called \textit{forward complete} if for every
$x_0\in\R^n$, $\sup I_{x_0}=\infty$. This is the case, for instance,
if $\frac{\partial f}{\partial x}$ is bounded.

\vspace{3mm}
\textit{Example~1.} Consider the system 
\begin{equation} \label{eq:simple_sys}
\dot{x} \m=\m x^3-x, \qquad x(0)=x_0\in\R,
\end{equation} 
whose trajectories are given by
\begin{equation}\label{eq:phi_ex1}
   \phi_t(x_0) \m=\m \frac{x_0}{\sqrt{(1-x_0^2)e^{2t}+x_0^2}}.
\end{equation}
All the trajectories starting from $x_0\in(-1,1)$ converge to $\{0\}$,
those starting from $x_0\in\{-1,1\}$ remain there, while those
starting from $|x_0|>1$ have finite escape time. Thus, $[-1,1]\subset
D_t$ for all $t\geq0$. On the other hand, when considering, e.g.,
$x_0=2$, we have that $x_0\in D_t$ only for $t<0.5\log(\tfrac{4}{3})$.

\begin{remark} \label{rmk:back_flow}
From the group property, if $-t\notin I_{x_0}$, then $\phi^{-1}_t
(\{x_0\})=\emptyset$. Therefore, for any set $A\subset\R^n$ and any 
$t\in\R$, we have \vspace{-1mm} 
\begin{equation} \label{eq:back_flow}
   \phi_t^{-1}(A) \m=\m \phi_{-t}(A\cap D_{-t}).
\end{equation}
\end{remark}

We show next two useful properties of the flow that allow us to
define the Perron-Frobenius semigroup in more general settings than,
e.g., \cite{Karabacak2018,Mauroy2013,Rajaram2015,Rajaram2013,
Rajaram2010}. The first one (Lemma~\ref{lmm:sem_flow}) is the group
property enjoyed by the pre-image of the flow, which is shown to hold
also when trajectories have finite escape time. The second one
(Lemma~\ref{lmm:flow_ns}) is the nonsingularity of the flow, which is
shown to hold when $f$ from \rfb{eq:ODE} is continuously
differentiable.

\begin{lemma} \label{lmm:sem_flow}
For any $A\subset\R^n$ and any $t,\tau\in\R$, we have \vspace{-1mm}
\begin{equation} \label{eq:sem_flow}
   \phi^{-1}_{t+\tau}(A) \m=\m \phi_\tau^{-1}(\phi_t^{-1}(A)).
\end{equation}
\end{lemma}

{\it Proof.} \m For any set $A\subset\R^n$ and any $t,\tau\in\R$ we have
from \rfb{eq:back_flow} that \vspace{-1mm}
\begin{equation*} \label{eq:1}
   \phi_{t+\tau}^{-1}(A) \m=\m \phi_{-(t+\tau)}(A\cap D_{-(t+\tau)}) 
   \m=\m\phi_{-\tau}(\phi_{-t}(A\cap D_{-(t+\tau)})).
\end{equation*}
Similarly, using again \rfb{eq:back_flow}, we can write \vspace{-1mm}
$$ \phi^{-1}_\tau(\phi_t^{-1}(A))=\phi_{-\tau}(\phi_{-t}(A\cap D_{-t})
   \cap D_{-\tau}).$$
The proof is completed observing that \vspace{-1mm}
$$ \phi_{-t}(A\cap D_{-(t+\tau)}) \m=\m \{x\in\R^n \ | \ \phi_t(x)\in 
   A \cap D_{-(t+\tau)}\}$$
$$ =\m \{x\in\R^n \ | \ \phi_t(x)\in A, \ x\in D_{-\tau}\} \m=\m 
   \phi_{-t}(A\cap D_{-t})\cap D_{-\tau},$$
where we have used that $\phi_t(x)\in D_{-t-\tau}$ implies that 
$x\in D_{-\tau}$. \hfill \qed

\begin{lemma} \label{lmm:flow_ns}
With the above notation, $\phi_t$ is nonsingular for every $t\in\R$,
i.e., if $N\in\Lscr(\R^n)$ such that $\mu(N)=0$, then
$\mu[\phi^{-1}_t(N)]=0$ for every $t\in\R$.
\end{lemma}

\textit{Proof.} \m From \rfb{eq:back_flow}, for any $N\in\mathcal{L}
(\R^n)$ and for every $t\in\R$, we have $\phi_t^{-1}(N)=\phi_{-t}(N
\cap D_{-t})$. Since $\phi_{-t}\in C^1(D_{-t};\R^n)$, using 
\cite[Lemma~7.25]{Rudin1987}, if $\mu(N)=0$ then $\mu[\phi^{-1}_t(N)]=
\mu[\phi_{-t}(N\cap D_{-t})]=0$. \hfill \qed

\section{Almost Global Attractivity} \label{sec:aGAS} 

Given a set $\L\subset\rline^n$, we denote $\L^c=\rline^n\setminus\L$.
For any $x\in\rline^n$ and $\L\subset\R^n$, $d(x,\L)=\inf_{z\in\L}
\|x-z\|$ is the distance from $x$ to $\L$. If $\L=\emptyset$, then 
$d(x,\L)=\infty$. For $\L\neq\emptyset$, $d(\cdot,\L):\R^n
\to [0,\infty)$ is continuous.

\begin{definition} \label{def:aGAS}
A closed set $\L\subset\rline^n$ is an \textit{invariant set} for 
the system \rfb{eq:ODE} if, for any $x\in\L$, $\phi_t(x)\in\L$ for 
all $t\geq 0$. An invariant set $\L$ is said to be a \textit{global 
attractor} for \rfb{eq:ODE} if for any $x_0\in \R^n$ we have 
$$ \lim_{t\to\infty} d(\phi_t(x_0),\L) \m=\m 0.$$
\end{definition}

\begin{definition} \label{def:aGAS_R}
An invariant set $\L\subset\rline^n$ is an \textit{almost 
global attractor} for the system \rfb{eq:ODE} if
\begin{equation} \label{eq:attr}
   \mu(\{x_0\in \R^n \ | \ \lim_{t\to\infty}d(\phi_t(x_0),\L)\neq0\})
   \m=\m 0 \m.
\end{equation}
\end{definition}

It follows from the above definitions that if, for instance, the
system \rfb{eq:ODE} is globally asymptotically stable with
equilibrium at $\{0\}$, a possible choice for a global attractor may
be $\L=\{0\}$ or any invariant set containing $0$. The same can be
said for almost global attractors, see Examples~2, 3 below.

\begin{remark}
The set $\L=\emptyset$ is an invariant set for the system
\rfb{eq:ODE}, but it cannot be a global attractor nor an almost global
attractor. If \rfb{eq:ODE} is forward complete, then the set
$\L=\R^n$ is a (trivial) global attractor.
\end{remark}

\begin{definition}
An invariant set $\L\subset\R^n$ is \textit{Lyapunov stable} for the
system \rfb{eq:ODE} if for any $\e>0$ there exists a $\delta>0$ such
that if $d(x_0,\L)<\delta$ then $d(\phi_t(x_0),\L)<\e$ for all 
$t\geq 0$.
\end{definition}

We point out that if $\L\subset\R^n$ is an almost global attractor for
the system \rfb{eq:ODE}, it does not follow that $\L$ is Lyapunov
stable. We present below two examples illustrating this fact: one
example obtained with a system similar to \rfb{eq:simple_sys}, and an
example built using the well-known Artstein's circles
\cite[Sect.~6]{Artstein1983}. (We mention that the latter example is
also studied in \cite[Example~2]{Monzon2006}.)

\vspace{3mm}
\textit{Example~2.} Consider the system 
\begin{equation}\label{eq:simple_sys2}
\dot{x} \m=\m -x^3+x, \qquad x(0)\in\R.
\end{equation} 
The following sets are invariant for \rfb{eq:simple_sys2}:
$\L_1=\{-1,1\}$, which is almost globally attractive and Lyapunov
stable; and $\L_2=\{-1,0,1\}$, which is globally attractive, but
not Lyapunov stable (since the equilibrium $\{0\}$ is unstable).

\begin{figure} 
\centering
\includegraphics[width=0.5\linewidth]{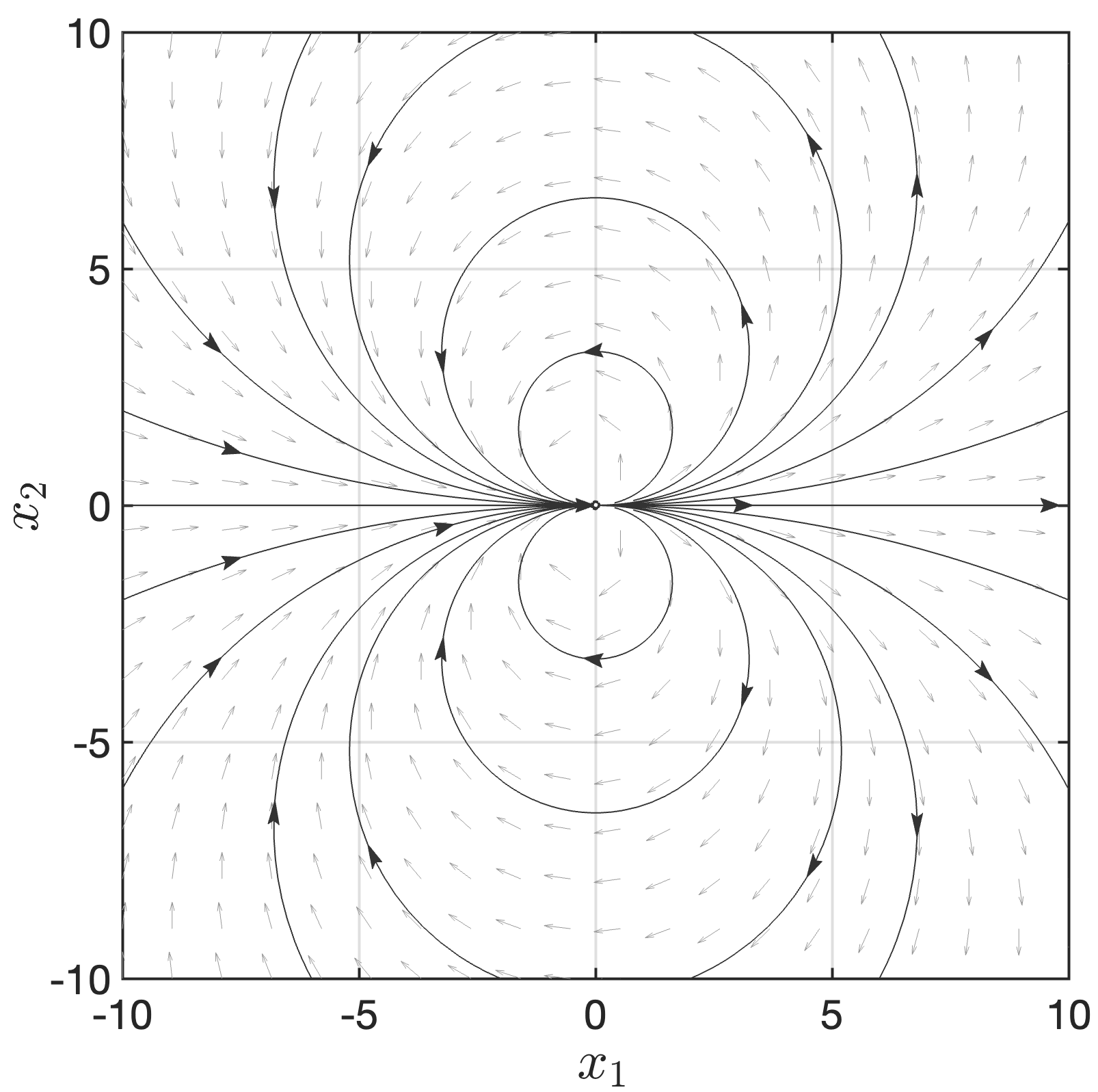}
\caption{The trajectories of the system \rfb{eq:counter}.}
\label{fig:counter}
\end{figure}

\vspace{3mm}
\textit{Example 3.} Consider the system
\begin{equation} \label{eq:counter}
   \begin{cases} \m\dot{x}_1\m=\m x_1^2-x_2^2, \\
   \m\dot{x}_2\m=\m 2x_1x_2, \end{cases}
   \qquad (x_1(0),x_2(0))\in\rline^2.
\end{equation}
When $x_2(0)\neq0$, the trajectories of \rfb{eq:counter} move along
circles, as shown in Figure~\ref{fig:counter}, while when $x_2(0)=0$,
$x_2(t)=0$ for all $t\geq 0$. The set $\L=\{0\}$ is an almost global
attractor for \rfb{eq:counter}, since all the trajectories converge
to $\L$, except those starting from the set $\Omega=\{(x_1(0),x_2(0))
\in\rline^2 \ | \ x_1(0)>0, \ x_2(0)=0\}$. Clearly, $\mu(\Omega)=0$.
The set $\L=\{0\}$ is not Lyapunov stable, since, as shown in
Figure~\ref{fig:counter}, any trajectory starting from $\Omega$
escapes to infinity. Similarly, any bounded invariant set containing
$\{0\}$ is not Lyapunov stable.

\begin{remark}
In \cite{Rajaram2015,Rajaram2013,Rajaram2010} an almost global
attractor is assumed to be Lyapunov stable, while our framework here
is more general.
\end{remark} 

\section{The Perron-Frobenius Semigroup} \label{sec:3} 

Intuitively, given a dynamical system on $\R^n$ and an invariant set
$\L\subset\R^n$, we can think of the corresponding Perron-Frobenius
semigroup as describing the evolution of the density of a mass of
water spread over $\L^c$, flowing along the trajectories of the
system, such that each water molecule moves according to the
differential equation of the system. Thus, if a trajectory escapes to
infinity or enters $\L$, the corresponding water molecule disappears
from $\L^c$, so that the total amount of water is a nonincreasing 
function of time. We give here a general definition of the
Perron-Frobenius semigroup and we show that it is a positive strongly
continuous semigroup of contractions. 

For $g\in C^1
(\rline^n)$ we denote $\nabla\cdot g \m=\m \frac{\partial
g_1}{\partial x_1}+ \frac{\partial g_2}{\partial x_2}+\dots + \frac{
\partial g_n}{\partial x_n}$. If $\L\subset\R^n$ is closed and $\rho
\in L^1(\L^c)$, then, by definition, $\rho(x)=0$ for all $x\in\L$, so
that $\rho\in L^1(\R^n)$. For any $\rho\in L^1(\L^c)$, we denote 
$\rho^+=\max(0,\rho(x))$, $\rho^-=\max(0,-\rho(x))$. Clearly, $\rho=
\rho^+-\rho^-$ and $|\rho|=\rho^++\rho^-$. We denote by $C_0^\infty
(\L^c)$ the test functions on $\L^c$, i.e., smooth functions with 
compact support on $\L^c$, and by $\Lscr(\L^c)$ the Lebesgue 
$\sigma$-algebra on $\L^c$. When integrating using the Lebesgue 
measure $\mu$, we simply write $\dd x$ in place of $\mu(\dd x)$.

\vspace{1mm}
The following standing assumption is kept throughout this section.

\vspace{2mm}
\noindent\textbf{Standing Assumption.} $\L\subset\rline^n$ is an 
invariant set for the system \rfb{eq:ODE}.

\begin{definition} \label{def:F-P} 
For any $t\geq 0$, we denote by $\tline_t: L^1(\L^c)\to L^1(\L^c)$ the
Perron-Frobenius operator corresponding to the flow of the system
\rfb{eq:ODE} on the open set $\L^c$, defined as 
follows:

For any $\rho\in L^1(\L^c)$, $\rho\geq 0$, and any $A\in\mathcal{L}
(\L^c)$, we define
\begin{equation} \label{eq:PF_def}
   \int_A[\tline_t\rho](x)\dd x \m=\m \int_{\phi_t^{-1}(A)}
   \rho(x)\dd x,
\end{equation}
so that $\tline_t\rho$ is defined as the Radon-Nikodym derivative of
the measure on the right-hand side of \rfb{eq:PF_def}. For any
$\rho\in L^1(\L^c)$, we decompose $\rho=\rho^+ -\rho^-$ (positive
and negative parts of $\rho$) and we define
\begin{equation} \label{eq:PF_pos_neg}
   \tline_t\rho \m=\m \tline_t\rho^+-\tline_t\rho^-.
\end{equation}
\end{definition}

\begin{remark}
As proved in Lemma~\ref{lmm:flow_ns}, the flow $\phi_t$ of the system
\rfb{eq:ODE} is nonsingular. Thus, it follows from Theorem
\ref{thm:Radon} that for $\rho\geq0$, $\tline_t\rho$ is well-defined
by \rfb{eq:PF_def}. In \cite{Rajaram2015, Rajaram2013,Rajaram2010} the
flow \textit{is assumed} to be nonsingular, although the vector field
$f$ from \rfb{eq:ODE} is assumed to be infinitely differentiable.
\end{remark}

\begin{remark} 
For some $x\in\rline^n$, the preimage $\phi_t^{-1}(\{x\})$ may be
empty. This is the case if the trajectory starting from $x$ escapes
backward in finite time, before reaching $-t$ (see Remark
\ref{rmk:back_flow}). However, this is not an issue for the definition
of $\tline_t$, since \rfb{eq:PF_def} is still well-defined.
\end{remark}

\begin{lemma}\label{lmm:PF_expl}
For every $\rho\in
L^1(\L^c)$ and for every $t\geq0$, we have that
\begin{equation} \label{eq:PF_D_t}
   \int_A[\tline_t\rho](x)\dd x \m=\m \int_{A\cap D_{-t}}[\tline_t
   \rho](x)\dd x \FORALL A\in\mathcal{L}(\L^c).
\end{equation}
We can compute $\tline_t\rho$ explicitly by 
\begin{equation} \label{eq:PF_expl}
   [\tline_t\rho](x) \m=\m \rho(\phi_{-t}(x))\det\frac{\partial
   \phi_{-t}(x)}{\partial x} \FORALL x\in A\cap D_{-t}, \ 
   A\in\mathcal{L}(\L^c).
\end{equation}
Moreover, we have 
\begin{equation} \label{eq:det_pos}
   \det\frac{\partial \phi_{-t}(x)}{\partial x} \geq 0 \FORALL x\in 
   A\cap D_{-t}, \ A\in\mathcal{L}(\L^c).
\end{equation}
\end{lemma}

\textit{Proof.} Without loss of generality, we assume $\rho\geq 0$.
(If this is not the case, we can write $\rho=\rho^+-\rho^-$ and use
\rfb{eq:PF_pos_neg}.) Using \rfb{eq:back_flow} and \rfb{eq:PF_def},
we have
$$ \int_A[\tline_t\rho](x)\dd x = \int_{\phi_{-t}(A\cap D_{-t})}
   \hspace{-1mm}\rho(x)\dd x=\int_{\phi_t^{-1}(A\cap D_{-t})}
   \hspace{-1mm}\rho(x)\dd x=\int_{A\cap D_{-t}}\hspace{-1mm}
   [\tline_t\rho](x)\dd x.$$
This proves \rfb{eq:PF_D_t}. To prove \rfb{eq:PF_expl} we use
\cite[Corollary~3.2.1]{Lasota1998} on the right-hand side of
\rfb{eq:PF_D_t}. (The backward flow $\phi_{-t}$ is well-defined on
$A\cap D_{-t}$.) Finally, it is clear from Definition~\ref{def:F-P}
that if $\rho\geq 0$ then also $\tline_t\rho\geq0$ (remember that the
Radon-Nykodym derivative is always $\geq 0$). For every $x\in\L^c$ we
can choose $\rho\in L^1(\L^c)$ such that $\rho(\phi_{-t}(x))>0$.
Hence, \rfb{eq:PF_expl} implies \rfb{eq:det_pos}. \hfill \qed

\vspace{3mm} The intuition behind equation \rfb{eq:PF_D_t} is the
following: The density $\tline_t\rho$ vanishes at points $x$ whose
trajectories escape (backward) in time $\tau\in[-t,0]$. Therefore, the
integral of $\tline_t\rho$ over the set $A\setminus D_{-t}$ is zero.

\begin{lemma}\label{lmm:abs}
For any $\rho\in
L^1(\L^c)$, we have that
\begin{equation} \label{eq:abs}
   \tline_t|\rho| \m=\m |\tline_t\rho| \FORALL t\geq 0.
\end{equation} 
\end{lemma}

\textit{Proof.} We start by proving that $(\tline_t\rho)^+=\tline_t
\rho^+$. For any $A\in\Lscr(\L^c)$,
$$ \int_A[\tline_t\rho]^+(x)\dd x \m=\m \int_A \max\{0,[\tline_t\rho]
   (x)\}\dd x \m=\m \int_{\tilde{A}}[\tline_t\rho](x)\dd x \m=\m
   \int_{\phi_t^{-1}(\tilde{A})}\rho(z)\dd z,$$
where $\tilde{A}=\{x\in A \ | \ [\tline_t\rho](x)\geq 0\}$. It 
follows from \rfb{eq:PF_expl} and \rfb{eq:det_pos} that 
$\tilde{A}=\{x\in A \ | \ \rho(\phi_t^{-1}(x))\geq0\}$, so that 
$$ \int_A[\tline_t\rho]^+(x)\dd x \m=\m \int_{\phi_t^{-1}(\tilde A)}
   \rho(z)\dd z\m=\m\int_{B}\rho(z)\dd z,$$
where $B=\{z\in\phi_t^{-1}(A) \ | \ \rho(z)>0\}=\{z\in\R^n \ | \ 
\phi_t(z)\in A, \ \rho(z)\geq0\}$. Therefore, 
$$ \int_A[\tline_t\rho]^+(x)\dd x\m=\m\int_{B}\rho^+(z)\dd z \m=\m
   \int_{\phi_t^{-1}(A)}\rho^+(z)\dd z \m=\m \int_A[\tline_t\rho^+]
   (x)\dd x.$$
Thus, $(\tline_t\rho)^+=\tline_t\rho^+$. Similarly, we can prove that
$(\tline_t\rho)^-=\tline_t\rho^-$. Hence, 
$|\tline_t\rho|=\tline_t|\rho|$. \hfill \qed

\vspace{3mm} We are now ready to prove the main result of this
section: the Perron-Frobenius operator defines a positive strongly
continuous semigroup of contractions. For more details on these
concepts, we refer to \cite{Engel2000}.

\begin{theorem}\label{thm:sem}
The family of Perron-Frobenius operators $\{\tline_t\}_{t\geq 0}$, as
defined in Definition~\ref{def:F-P}, has the following properties:
\begin{itemize}
\item[\rm (a)] $\{\tline_t\}_{t\geq0}$ is a \textit{strongly
continuous semigroup}, i.e.,
\begin{itemize}
\item[\rm (i)] $\tline_0=\mathrm{I}$, where $\mathrm{I}$ denotes the 
identity on $L^1(\L^c)$.
\item[\rm (ii)] $\tline_{t+\tau}=\tline_t\tline_\tau$ for every $t,
\tau\geq0$.
\item[\rm (iii)] $\lim_{t\to0, \m t\geq 0}\tline_t\rho=\rho$ for all 
$\rho\in L^1(\L^c)$.
\end{itemize}
\item[\rm (b)] $\{\tline_t\}_{t\geq0}$ is a \textit{semigroup of
contractions}, i.e., for all $t\geq 0$,
\begin{equation} \label{eq:PF_contraction}
   \|\tline_t\rho\|_{L^1(\L^c)} \leq \|\rho\|_{L^1(\L^c)} \FORALL 
   \rho\in L^1(\L^c).
\end{equation}
\item[\rm (c)] $\{\tline_t\}_{t\geq0}$ is a \textit{positive 
semigroup}, i.e., for all $t\geq 0$, 
\begin{equation} \label{eq:PF_positive}
   \tline_t\rho\geq 0 \FORALL \rho\in L^1(\L^c), \ \rho\geq 0.
\end{equation}
\end{itemize}
\end{theorem}

We mention that Lemma~\ref{lmm:abs} is a much stronger statement than
\rfb{eq:PF_positive}.

\vspace{2mm} \textit{Proof.} (a) Without loss of generality, we assume
that $\rho\geq0$. (If this is not the case, we can simply write
$\rho=\rho^+-\rho^-$ and then use \rfb{eq:PF_pos_neg}.)

(a)(i) As discussed at the beginning of Subsection~\ref{sec:prob},
for all $x_0\in\R^n$ we have $0\in I_{x_0}$ and
$x(0)=\phi_0(x_0)=x_0$.  Clearly, for any $A\subset\R^n$,
$\phi_0^{-1}(A)=A$.  Therefore, for every non-negative $\rho\in
L^1(\L^c)$ and any $A\in\Lscr(\L^c)$, we have
$$ \int_A[\tline_0\rho](x)\dd x \m=\m \int_{\phi_0^{-1}(A)}\rho(x)\dd x
   \m=\m \int_A\rho(x)\dd x.$$
From the above, it follows that $\tline_0\rho=\rho$.

(a)(ii) Let $t,\tau\geq0$. Using \rfb{eq:sem_flow}, for any $A\in
\Lscr(\L^c)$ we can write
$$ \int_A[\tline_{t+\tau}\rho](x)\dd x \m=\m \int_{\phi^{-1}_{t+\tau}
   (A)}\rho(x)\dd x=\int_{\phi_\tau^{-1}(\phi_t^{-1}(A))}\rho(x)\dd x$$
$$ =\m \int_{\phi^{-1}_t(A)}[\tline_\tau\rho](x)\dd x=\int_{A}
   [\tline_t\tline_\tau\rho](x)\dd x.$$
Therefore, we have that  $\tline_{t+\tau}=\tline_t\tline_\tau$ for 
every $t,\tau\geq 0$. 

Since we need (b) to prove (a)(iii), first we prove (b).  

(b) From \rfb{eq:PF_def} and \rfb{eq:abs}, for every $\rho\in L^1
(\L^c)$ we have the following:
\begin{equation} \label{eq:contr_1}
   \|\tline_t\rho\|_{L^1(\L^c)} \m= \int_{\L^c}|[\tline_t\rho](x)|\m
   \dd x = \int_{\L^c}[\tline_t|\rho|](x)\m\dd x = \int_{\phi^{-1}_t
   (\L^c)}|\rho(x)|\dd x.
\end{equation}
By assumption, the set $\L\subset\R^n$ is invariant, therefore, for 
every $t\geq 0$, $\phi^{-1}_t(\L^c)\subset\L^c$. Using the above 
inclusion in \rfb{eq:contr_1}, for every $t\geq0$ we get that 
$$ \|\tline_t\rho\|_{L^1(\L^c)} \m\leq\m \int_{\L^c}|\rho(x)|\dd x
   \m=\m\|\rho\|_{L^1(\L^c)},$$
for all $\rho\in L^1(\L^c)$. Thus, $\{\tline_t\}_{t\geq 0}$ is a 
semigroup of contractions.

(a)(iii) Let $\psi\in C_0^\infty(\L^c)$, i.e., $\psi$ is a test
function on $\L^c$. Recall the notation $D_t$ after \rfb{eq:ODE}. The
open sets $D_1\subset D_{\half}\subset D_{\frac{1}{3}}\subset
D_{\frac{1}{4}}\subset\dots$ are a covering of $\L^c$, because from
each point $x\in\L^c$ the flow $\phi_t(x)$ is defined for some $t>0$.
Hence, the sequence $\{D_\frac{1}{k}\}_{k\in\nline}$ is an open
covering of the support of $\psi$, $\mathrm{supp}\m\psi$. Since
$\mathrm{supp}\m\psi$ is compact, there exists $k_0\in\nline$ such
that $\mathrm{supp}\m\psi\subset D_{\frac{1}{k_0}}$. The function
$(t,x)\to\phi_t(x)$ is continuous on the compact set $[0,k_0^{-1}]
\times\mathrm{supp}\m\psi$, hence the image of this set,
$$ K \m=\m \{\phi_t(x) \ | \ t\in\left[0,k_0^{-1}\right], \ x\in
   \mathrm{supp}\m\psi\}$$
is compact. For any $t\in[0,k_0^{-1}]$, using \rfb{eq:PF_D_t} and 
\rfb{eq:PF_expl}, we have
$$ \|\tline_t\psi-\psi\|_{L^1(\L^c)} \m=\m \int_{\L^c}|[\tline_t\psi]
   (x)-\psi(x)|\m\dd x \m=\m \int_{\L^c\cap D_{-t}}|[\tline_t\psi](x)
   -\psi(x)|\m\dd x$$
\begin{equation} \label{eq:str_cont_1}
   =\m \int_{\L^c\cap D_{-t}}\left|\m\psi(\phi_{-t}(x))\det\frac
   {\partial \phi_{-t}(x)}{\partial x}-\psi(x)\m\right|\m\dd x.
\end{equation}
Note that if the expression in the last integral is non-zero, then we
must have $x\in K$. Hence, \rfb{eq:str_cont_1} implies that
\begin{equation} \label{eq:srt_cont_2}
   \|\tline_t\psi-\psi\|_{L^1(\L^c)} \m=\m \int_K\left|\m\psi
   (\phi_{-t}(x))\det\frac{\partial\phi_{-t}(x)}{\partial x}-\psi(x)
   \m\right|\dd x.
\end{equation}
Since $\psi\in C_0^\infty(\L^c)$, $\phi_{-t}\in C^1(D_{-t};\R^n)$, 
and $\phi_0(x)=x$, it follows that
\begin{equation} \label{eq:cont_Jac}
   \lim_{t\to 0,\m t\geq 0} \psi(\phi_{-t}(x))\det\frac{\partial 
   \phi_{-t}(x)}{\partial x}\m=\m\psi(x),
\end{equation}
uniformly with respect to $x\in K$. Therefore, combining 
\rfb{eq:srt_cont_2} and \rfb{eq:cont_Jac}, we get that
$$ \lim_{t\to0,\m t\geq0} \|\tline_t\psi-\psi\|_{L^1(\L^c)} \m=\m 0 
   \FORALL \psi\in C_0^\infty(\L^c),$$
since the integrals are over the compact set $K$. The set $C_0^\infty
(\L^c)\subset L^1(\L^c)$ is dense, thus, using (b), we have that 
$\lim_{t\to 0,\m t\geq 0}\tline_t\rho=\rho$ for all $\rho\in 
L^1(\L^c)$.

(c) For any $\rho\in L^1(\L^c)$, $\rho\geq 0$, we have $\rho=|\rho|$.
Thus, from \rfb{eq:abs}, we get
$$ \tline_t\rho \m=\m \tline_t|\rho|\m=\m|\tline\rho|\geq0 \FORALL 
   \rho\in L^1(\L^c), \ \rho\geq 0.$$
This proves that the operators $\{\tline_t\}_{t\geq0}$ are positive.
\hfill \qed

\begin{remark}
In \cite[Sect.~7.4]{Lasota1998} the Perron-Frobenius semigroup is
defined for a general semidynamical system, which evolves on a
Hausdorff space $X$, and whose trajectories are defined in $X$ for all
$t\geq0$. Thus, their framework is more general in some respects and
less general in others. Under these assumptions, they prove some of
the properties in Theorem~\ref{thm:sem}, except for strong continuity
and contraction. The strong continuity property is proved in
\cite[Sec.~7.6]{Lasota1998} assuming that the semidynamical system is
the flow of a system as in \rfb{eq:ODE} with $D_t=\R^n$ for all
$t\geq0$.  The strong continuity of the Perron-Frobenius semigroup is
discussed in \cite{Navas2002}, in a different framework.
\end{remark}

\begin{remark}
In this remark we assume that $\L=\emptyset$ in Definition
\ref{def:F-P}, so that $\tline_t:L^1(\R^n)\to L^1(\R^n)$. Since
$\L=\emptyset$ is invariant for \rfb{eq:ODE} also in backward time, we
can define the {\em backward Perron-Frobenius semigroup} corresponding
to \rfb{eq:ODE}, for $t\leq 0$ (again $\tline_t:L^1(\R^n)\to L^1(\R^n)
$) as in Definition~\ref{def:F-P}. It is tempting to think that
$\tline_{-t}=\tline_t^{-1}$, but, in general, this is not the case.
Indeed, if there exists a set $M\in\mathcal{L}(\R^n)$, with
$\mu(M)>0$, and a $t>0$ such that $M\cap D_t=\emptyset$, then for any
$\rho\in L^1(\R^n)$ that is supported on $M$, we have that
$\tline_t\rho=0$ (this follows from \rfb{eq:PF_D_t}). Thus, $\tline_t$
is not invertible and cannot be extended to a group. In the case of
$D_t=\R^n$ for one $t>0$ (and hence for all $t\geq0$) then $\tline_t$
is \textit{isometric}, i.e., $\|\tline_t\rho\|=\|\rho\|$ for each
$\rho\in L^1(\R^n)$ \cite[eq.~(7.4.4)]{Lasota1998}. Similarly, if
$D_t=\R^n$ for one $t<0$ (and hence for all $t\leq0$), then the
backward Perron-Frobenius semigroup is isometric. If both semigroups
(forward and backward) are isometric, then $\tline_{-t}=\tline_t^{-1}$
is true and the Perron-Frobenius semigroup can be extended to a group
of isometric operators on $L^1(\R^n)$.
\end{remark}

\begin{remark}
We think that the generator of the semigroup $\tline_t$ from
Definition~\ref{def:F-P} is \vspace{-2mm}
\begin{equation} \label{eq:gen}
   \mathbb{A}:\Dscr(\mathbb{A})\to L^1(\L^c),\qquad 
   \mathbb{A} \rho\m=\m-\nabla\cdot(f\rho),
\end{equation}
where 
$$ \Dscr(\mathbb{A}) \m=\m \{\rho\in L^1(\L^c) \ | \ \nabla
   \cdot(f\rho)\in L^1(\L^c)\}.$$

A proof of \rfb{eq:gen}, without explicit description of
$\Dscr(\mathbb{A})$, is given in \cite[Section~7.6]{Lasota1998}. Their
proof exploits the duality between the Koopman operator and the
Perron-Frobenius operator. A correct description of
$\Dscr(\mathbb{A})$ seems to be lacking in the literature. We will not
investigate deeper questions about the generator here, since it is not
needed in this paper.

The expression $\nabla\cdot(f\rho)$ in the above formula is defined in
the sense of distributions, i.e.,
$$ \int_{\L^c}\nabla\cdot(f\rho)\varphi\m\dd x \m=\m -\int_{\L^c} f
   \rho \cdot (\nabla \varphi) \m\dd x \FORALL \varphi\in 
   C_0^\infty(\L^c).$$
Since $f\rho\in L^1_{\rm loc}(\L^c)$ for any $\rho\in L^1(\L^c)$ and
any $f\in C^1(\R^n;\R^n)$, the last integral is well-defined. Spaces
similar to our $\Dscr(\mathbb{A})$ are defined in \cite
[Ch.~9]{Dautray1990}. The domain $\Dscr(\mathbb{A})$ is not given in
\cite{Karabacak2018,Lasota1998,Vaidya2008} and it is given
inaccurately in \cite{Rajaram2015,Rajaram2013,Rajaram2010}, since it
is required that $\rho\in L^1(X)$ and $\nabla\rho\in L^1(X)$, which
leads to a non-closed operator $\mathbb{A}$. (Here $X\subset\R^n$ is a
forward invariant set for \rfb{eq:ODE}, whose existence is required in
\cite{Rajaram2015,Rajaram2013,Rajaram2010}.)
\end{remark}

\section{Main Result} \label{sec:4} 

We prove that the strong stability of the Perron-Frobenius semigroup
on certain domains is equivalent to the set $\L\subset \R^n$ being an
almost global attractor in the sense of Definition~\ref{def:aGAS_R}. A
key step for our proof is to show that, given an almost global
attractor $\L$, in any $\e$-neighborhood of $\L$ we can find an almost
global attractor in finite time, as shown in
Proposition~\ref{prop_L_e}.

\begin{definition} \label{def:aGAS_ft}
An invariant set $\L\subset\rline^n$ is an \textit{almost 
global attractor in finite time} for the system \rfb{eq:ODE} if,
for almost every initial state $x_0\in \R^n$, there exists
a $T\geq0$ such that $\phi_{T}(x_0)\in\L$.
\end{definition}

\begin{proposition}\label{prop_L_e}
Let $\L\subset\R^n$ be an almost global attractor for the system
\rfb{eq:ODE}. Then, for every $\e>0$ the set
\begin{equation} \label{eq:aGAS_ft}
   \L_\e \m=\m \{x\in \R^n \ | \ d(\phi_t(x),\L) \leq\e \ \forall \ 
   t\geq0\}.
\end{equation}
is an almost global attractor in finite time for the system
\rfb{eq:ODE}.
\end{proposition}

\textit{Proof.} From \rfb{eq:aGAS_ft}, $\L_\e$ is closed and invariant
for \rfb{eq:ODE}. Let $x_0\in \R^n$ such that
$\lim_{t\to\infty}d(\phi_t(x_0),\L)=0$ (this holds for almost every
$x_0\in \R^n$). Then, from the definition of limit, for every $\e>0$
there exists $T\geq0$ such that $d(\phi_t(x_0),\L)\leq\e$ for all
$t\geq T$. It follows from the group property of the flow $\phi$ and
\rfb{eq:aGAS_ft} that $\phi_T(x_0)\in\L_\e$.  This shows that indeed
$\L_\e$ is an almost global attractor in finite time for
\rfb{eq:ODE}. \hfill \qed

\begin{figure} 
\centering \includegraphics[width=0.5\linewidth]{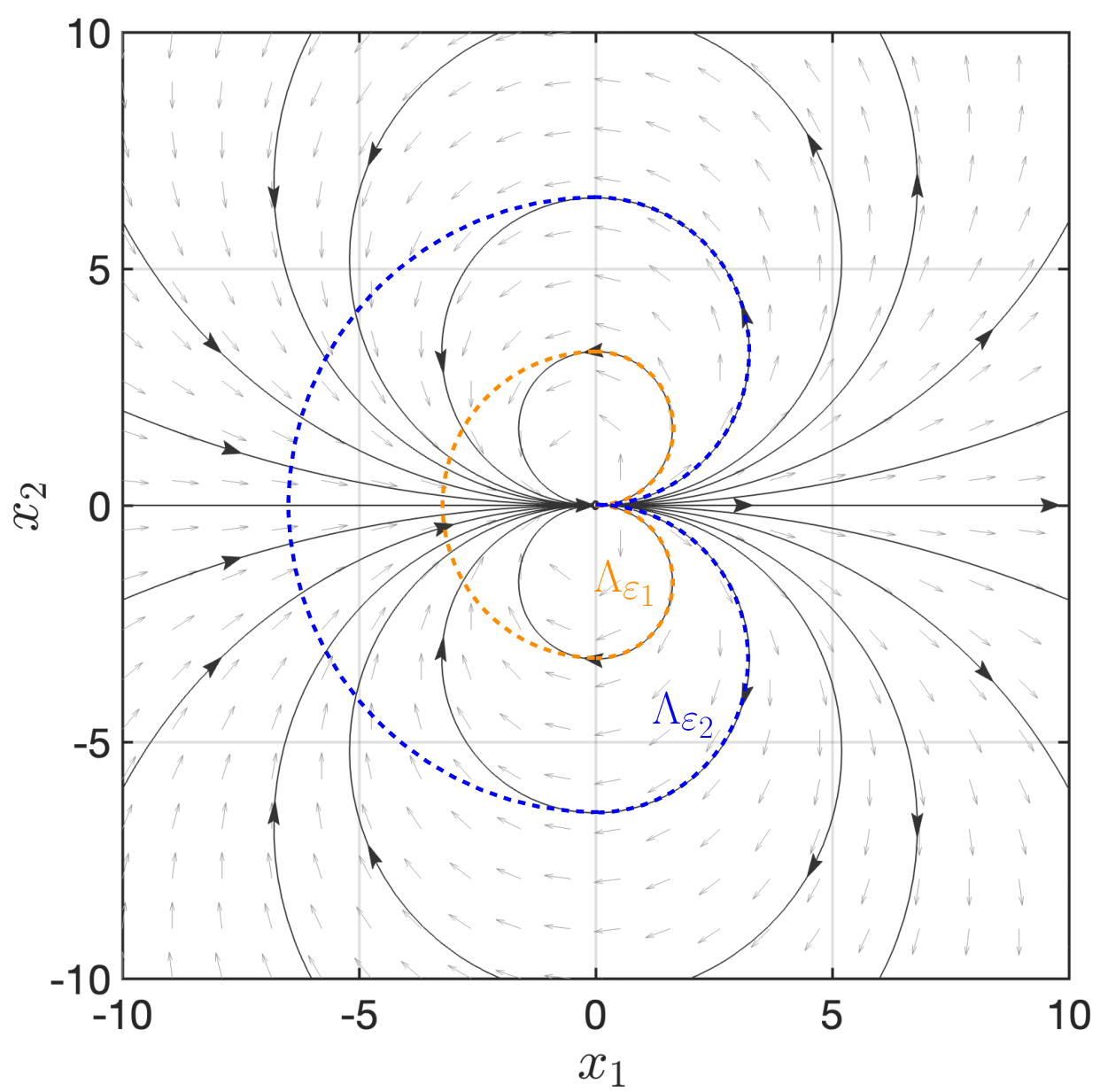}
\caption{Two possible choices of $\L_\e$ from \rfb{eq:aGAS_ft} for the 
system \rfb{eq:counter} (their boundaries are shown in orange and in
blue).} \label{fig:hearts}
\end{figure}

\vspace{3mm} \textit{Example~3 (cont'd).} Recall that the system
\rfb{eq:counter} has an almost global attractor set $\L=\{0\}$. In 
Figure~\ref{fig:hearts} we show two possible choices for an almost 
global attractor in finite time $\L_\e\subset\R^n$ for 
\rfb{eq:counter}, derived as in \rfb{eq:aGAS_ft}.

\vspace{3mm} We state below our main result. Recall from
Subsection~\ref{sec:prob} that, given a dynamical system as in
\rfb{eq:ODE}, for every $x_0\in\R^n$ the set $I_{x_0}\subset\R$
indicates the maximal open interval (in time) where $\phi_t(x_0)$ is 
defined. Therefore, the statement ``$\sup I_{x_0}<\infty$'' is
equivalent to saying ``the trajectory of \rfb{eq:ODE} starting from
the initial condition $x_0$ escapes in finite time''.

\begin{theorem}\label{thm:main}
For a dynamical system described by \rfb{eq:ODE} and a non-empty
invariant set $\L\subset\rline^n$, the following statements are 
equivalent: \vspace{-1mm}
\begin{itemize}
\item[\rm (i)]  For almost every initial condition $x_0\in\R^n$, 
\vspace{-2mm}
$$ \text{either} \qquad \sup I_{x_0}<\infty \qquad \text{or} \qquad 
   \lim_{t\to\infty}d(\phi_t(x_0),\L)=0.$$
\item[\rm (ii)] For every $\e>0$ the Perron-Frobenius semigroup 
$\tline_t:L^1(\L_\e^c)\to L^1(\L_\e^c)$, with $\L_\e$ as in 
\rfb{eq:aGAS_ft}, is strongly stable, i.e., 
\begin{equation} \label{eq:PF_str_st}
   \lim_{t\to\infty}\tline_t\rho \m=\m 0 \FORALL \rho\in L^1(\L_\e^c).
\end{equation}
\end{itemize}
\end{theorem}

\textit{Proof.} Assume (i). Let $\e>0$, and $\L_\e$ as in 
\rfb{eq:aGAS_ft}. For all $k\in\N$, denote \m $E_k \m=\m \{\m x_0\in
\R^n \ | \ \sup I_{x_0}<k \m\}$, 
\begin{equation} \label{eq:N_e}
   A_k \m=\m (\phi_k^{-1}(\L_\e)\cup E_k)\cap\L_\e^c,\qquad\ N_\e 
   \m=\m  \L_\e^c\setminus\bigcup_{k\in\N}A_k.
\end{equation}
Thus, $\bigcup_{k\in\N}A_k$ consists of those $x_0\in\L_\e^c$ for
which the trajectory $\phi_t(x_0)$ (for $t>0$) either blows up in
finite time or it enters $\L_\e$. From statement~(i) here, we have
$\mu(N_\e)=0$. Clearly $A_k\subset A_{k+1}\subset\L_\e^c$ for all
$k\in\N$. For any $k\in\N$ we introduce the set
$$ \mathcal{D}_k \m=\m \{\m\rho\in L^1(\L_\e^c) \ | \ \rho(x)=0 \ 
   \text{for} \ x\in \L_\e^c\setminus A_k \m\}.$$
We claim that $\tline_k\mathcal{D}_k=\{0\}$, for all $k\in\N$. Let
$A\subset\L_\e^c$, $A\in\mathcal{L}(\L^c)$, and $\rho\in\mathcal{D}_k$.
Then, using \rfb{eq:PF_def} and \rfb{eq:abs}, for all $k\in\N$ we
can write
$$ \left| \m\int_A[\tline_k\rho](x)\dd x \m\right| \leq\m \int_A| 
   [\tline_k\rho](x)|\dd x = \int_{\phi_k^{-1}(A)}|\rho(x)|\dd x
   \m= \int_{\phi_k^{-1}(A)\cap A_k}|\rho(x)|\dd x.$$
In the last step we have used that $\phi_k^{-1}(A)\subset\L_\e^c$.
Finally, since $\phi_k^{-1}(A)\subset\phi_k^{-1}(\L_\e^c)$ and
$\phi_k^{-1}(\L_\e^c)\cap A_k=\emptyset$ (by the definition of $A_k$),
$$ \left| \m\int_A[\tline_k\rho](x)\dd x \m\right| \m\leq\m 
   \int_{\phi_k^{-1}(A)\cap A_k} |\rho(x)|\dd x \m=\m 0.$$
Thus, $\tline_k\Dscr_k=\{0\}$, for all $k\in\N$. We claim that 
$\cup_{k\in\N}\Dscr_k$ is dense in $L^1(\L_\e^c)$. Let $\rho\in L^1
(\L_\e^c)$ and $\rho_k=\rho|_{A_k}$, the restriction of $\rho$ to 
$A_k$, for all $k\in\N$. Then
$$ \|\rho-\rho_k\|_{L^1(\L_\e^c)} \m=\m \int_{\L_\e^c}|\rho(x)-
   \rho_k(x)|\dd x \m=\m \int_{\L_\e^c\setminus A_k}|\rho(x)|\dd x.$$
The sequence of sets $\{A_\e^c \setminus A_k\}_{k\in\nline}$ is 
decreasing, and its intersection is the set $N_\e$ from \rfb{eq:N_e}.
Hence, by elementary integration theory,
$$ \lim_{k\to\infty}\|\rho-\rho_k\|_{L^1(\L_\e^c)} \m=\m \int_{N_\e}
   |\rho(x)|\dd x \m=\m 0,$$
proving that indeed $\cup_{k\in\N}\Dscr_k$ is dense in $L^1
(\L_\e^c)$. According to Theorem~\ref{thm:sem}, $\tline_k$ is a
contraction for all $k\in\N$. Thus, from $\cup_{k\in\N}\Dscr_k$ being
dense in $L^1(\L_\e^c)$, we have that $\lim_{k\to\infty}\tline_k\rho
=0$ for all $\rho\in L^1(\L_\e^c)$. This proves statement (ii).

Assume (ii). We denote $B_\delta=\{x\in\rline^n \ | \ \norm{x}<\delta
\}$, for any $\delta>0$. Let $\e>0$. We claim that $\mu(N_\e)=0$, with
$N_\e$ as in \rfb{eq:N_e}. We argue by contradiction. Assume that
$\mu(N_\e)>0$. Then, there exists a $\delta>0$ such that $\mu(N_\e\cap
B_\delta)>0$. Let $\rho$ be the characteristic function of $N_\e\cap
B_\delta$, so that $\rho\geq0$, $\rho\in L^1(\L_\e^c)$ and
$\|\rho\|>0$. If $x\in N_\e$, then $\phi_k(x)\notin\L_\e$ for all
$k\in\N$ and $x\notin E_k$ for all $k\in\N$. Thus, it must be that
$\phi_t(x)\in\L_\e^c$ for all $t\geq0$. Hence,
$\phi_t^{-1}(\L_\e^c)\supset N_\e$ for all $t\geq0$. Then,
$$ \|\tline_t\rho\|_{L^1(\L_\e^c)} \m=\m \int_{\phi_t^{-1}
   (\L_\e^c)}\rho(x)\dd x\m=\m \int_{N_\e}\rho(x)\dd x \m=\m
   \|\rho\|_{L^1(\L_\e^c)} \FORALL t \geq 0.$$
The above equality contradicts \rfb{eq:PF_str_st}. Therefore,
$\mu(N_\e)=0$, meaning that, for almost every $x_0\in\R^n$, either
$\sup I_{x_0}<\infty$ or there exists a $T\geq0$ such that
$\phi_{T}(x_0)\in\L_\e$.  This argument is valid for every $\e>0$. Let
$\e_m=\frac{1}{m}$, $m\in\N$. Clearly, $N_{\e_2}\subset N_{\e_1}$
for $\e_1<\e_2$. From the definitions of $N_\e$ and $A_k$, we have
$$ N = \bigcup_{m\in\N}N_{\e_m}=\bigcup_{m\in\N}\left(\L_{\e_m}^{c}
   \setminus\bigcup_{k\in\N}\left\{x\in\L_{\e_m}^{c} \ | \ \phi_k(x)
   \in\L_{\e_m} \ \text{or} \ \sup I_{x}<k\right\}\right).$$
From $\mu(N_{\e_m})=0$ for all $m\in\N$, it follows that $\mu(N)=0$.
Therefore, for every $x\in\L^{c}\setminus N$, we have that
$$ \text{either} \quad \sup I_x<\infty \quad \text{or} \quad 
   \limsup_{t\to\infty} d(\phi_t(x),\L)\leq\frac{1}{m} \quad 
   \forall \ m\in\N.$$ 
The above, combined with $\mu(N)=0$, implies statement (i). \hfill\qed

\begin{remark}
It is natural to ask why we need to introduce the sets $\L_\e$ for
Theorem~\ref{thm:main}, could we maybe work with $\L$ in place of
$\L_\e$? Unfortunately, this is not the case. Consider on $\R^n$ the
trivial system $\dot{x}=-x$, so that $\phi_t(x_0)=e^{-t}x_0$. The set
$\L=\{0\}$ is a global attractor, but on $L^1(\L^c)$ the
Perron-Frobenius semigroup is isometric and, thus, not strongly
stable.
\end{remark}

We present below a simple one-dimensional nonlinear system for which
many trajectories have finite escape time (i.e., $\sup
I_{x_0}<\infty$), yet the associated Perron-Frobenius semigroup is
strongly stable. (This example highlights the importance of allowing
finite escape time in statement~(i) of Theorem~\ref{thm:main}.)

\vspace{3mm}
\textit{Example~1 (cont'd).
As discussed after Definition~\ref{def:aGAS_R}, there are, in general,
many possible choices for an almost global attractor $\L$.  For
instance, for the system \rfb{eq:simple_sys} we could choose any set
of the form
\begin{equation} \label{eq:L_simple_sys}
\L \m=\m [-\delta,\delta] \qquad \text{for any} \ 
   \delta\in[0,1].
\end{equation}
Then, using the implication (i)$\Rightarrow$(ii) of Theorem
\ref{thm:main}, it follows that $\tline_t:L^1(\L_\e^c)\to L^1
(\L_\e^c)$ associated to the flow of the dynamical system
\rfb{eq:simple_sys} is strongly stable for every $\e>0$.
Alternatively, assuming that we wanted to investigate the almost
global attractivity property of $\L$ (without knowing it \textit{a
priori}), we could proceed as follows. From \rfb{eq:phi_ex1}, we can
compute
\begin{equation*}
   \det\frac{\partial\phi_{-t}(x)}{\partial x}=\frac{\sqrt{(1-x^2)
   e^{-2t}+x^2}-x^2\frac{1-e^{-2t}}{\sqrt{(1-x^2)e^{-2t}+x^2}}}
   {(1-x^2)e^{-2t}+x^2}.
\end{equation*}
Using the above in \rfb{eq:PF_expl}, an explicit expression for the
Perron-Frobenius semigroup associated to the flow of the system
\rfb{eq:simple_sys} can be obtained. At this point, condition
\rfb{eq:PF_str_st} can be checked and the implication
(ii)$\Rightarrow$(i) of Theorem~\ref{thm:main} can be used to infer
almost global attractivity of $\L$ from \rfb{eq:L_simple_sys}.}

\begin{definition} \label{def:for_inv}
A system \rfb{eq:ODE} is \textit{almost forward complete} if
$$\mu(\{x_0\in \rline^n \ | \ \sup I_{x_0}<\infty\}) \m=\m 0,$$
i.e., for almost every $x_0\in\R^n$ the flow $\phi_t(x_0)$ is defined
for all $t\geq 0$.
\end{definition}

\begin{corollary} \label{cor:main}
Given an almost forward complete system \rfb{eq:ODE} and a non-empty
invariant set $\L\subset\rline^n$ for \rfb{eq:ODE}, the following
are equivalent:
\begin{itemize}
\item[\rm (i)] $\L$ is an almost global attractor for the system 
\rfb{eq:ODE}.
\item[\rm (ii)] For every $\e>0$ the Perron-Frobenius semigroup 
$\tline_t:L^1(\L_\e^c)\to L^1(\L_\e^c)$, with $\L_\e$ as in 
\rfb{eq:aGAS_ft}, is strongly stable
\end{itemize}
\end{corollary}

This corollary follows easily from Theorem~\ref{thm:main}.

\vspace{3mm} \textit{Example~3 (cont'd).} The system \rfb{eq:counter}
is almost forward complete since $\mu(\{x_0\in \rline^2 \ | \ \sup
I_{x_0}<\infty\})=\mu(\Omega)=0$. (Here $\Omega\subset\R^n$ is the set
introduced after \rfb{eq:counter}.) Thus, choosing, e.g., $\L=\{0\}$,
the Perron-Frobenius semigroup $\tline_t:L^1(\L_\e^c)\to L^1(\L_\e^c)$
associated to \rfb{eq:counter} is strongly stable for every $\e>0$.

\begin{remark} \label{Herzliya}
There is a connection between our results and the famous density
function theory of A.~Rantzer \cite{Rantzer2001,Rantzer2002}. We
only outline the idea, without giving rigorous arguments. Suppose
that the assumptions and condition (ii) of Corollary \ref{cor:main} 
are satisfied with $\L=\{0\}$. (According to this corollary, also 
condition (i) is satisfied.) Let $b\in L^1(\rline^n)$ be an almost
everywhere strictly positive and continuous function on $\L^c$ such
that for any $\e>0$, denoting by $\tline_t$ the Perron-Frobenius 
semigroup on $L^1(\L^c_\e)$, the limit \vspace{-1mm}
\begin{equation} \label{Sinwar}
   R_0 b \m=\m \lim_{T\rarrow\infty} \int_0^T \tline_t b \m\dd t
\end{equation}
exists in $L^1(\L^c_\e)$. The existence of such $b$ can be derived
from condition (ii). Such a function $b$ has to be zero on any closed
invariant set contained in $\L^c$ (such sets must have measure zero).
In fact, \rfb{Sinwar} defines an almost everywhere strictly positive 
function $\rho$ on all of $\L^c$, because every point in $\L^c$ is 
contained in some set $\L^c_\e$ with $\e>0$, and we impose that the 
restriction of $\rho$ to $\L^c_\e$ is $R_0 b$. The function $\rho$
does not have to be in $L^1(\L^c)$, but of course its restriction to 
the exterior of any ball of radius $\e>0$ is integrable. Denoting the
generator of $\tline_t$ by $\mathbb{A}$, we know that for any complex 
number $s$ with $\Re s>0$, \vspace{-2mm}
$$ R_s b \m=\m (sI-\mathbb{A})^{-1}b \m=\m \lim_{T\rarrow\infty} 
   \int_0^T e^{-st} \tline_t b \m\dd t \m=\m \int_0^\infty e^{-st} 
   \tline_t b \m\dd t \m.$$
Clearly $R_s b$ satisfies $(sI-\mathbb{A})R_s b = b$. Taking limits as 
$s\rarrow 0$, we get that $-\mathbb{A}\rho=b$, which implies Rantzer's 
inequality: $-\mathbb{A}\rho>0$ almost everywhere on $\L^c$, or more 
explicitly, $\nabla(f\rho)>0$ almost everywhere on $\L^c$. 

A challenge is to prove the implication in the converse direction, 
i.e., to derive our condition (ii) from Rantzer's inequality, and to
do all this for an arbitrary closed invariant set $\L$.
\end{remark}

\section{Concluding remarks} \label{sec:5} 

We have shown the equivalence between the strong stability of certain
Perron-Frobenius semigroups and the almost global attractivity
property of invariant sets for nonlinear dynamical systems. To this
aim, we constructed a more general framework than what was available
in the literature: We do not assume that the flow of the system is
nonsingular (rather, we show that it is the case for a $C^1$ vector
field), we do not assume the existence of a forward invariant compact
set, and we do not assume local stability of the attractor. We rely
only on basic notions from dynamical systems and measure theory.

Our stability result is intended as a first step towards a more
comprehensive theory for operator-based tools for the analysis and
control of nonlinear systems, via the connection with the theory of
density functions of A.~Rantzer (see Remark \ref{Herzliya}). In
particular, we hope that possible applications may stem from power
systems, since, as shown in, e.g., \cite{Bara2017,LorRei24,
Natarajan2018,Weiss2019}, a synchronous generator can be modeled (in
its simplest instance) as a damped pendulum, which, in turn, has an
almost global attractors. For this, the equivalent characterization
proposed in Theorem~4.3 could inspire numerical algorithms, e.g., in
the spirit of those developed in \cite{Rajaram2010}, to study the
almost global stability properties of synchronous generators. This is
particularly relevant since, as discussed in \cite{Angeli2003},
studying the almost global attractivity properties of saddle points
(e.g., stable equilibrium points of damped pendulums) using density
functions is not recommended, so that we hope to find a more suitable
approach using operator-based tools.

Future works include extending the results from \cite{Rajaram2010} to
our more general framework, and possible applications of our theory to
solve control problems for systems with almost global attractors. For
instance, given a dynamical system $\dot{x}=f(x,u)$, with $x\in\R^n$,
$u\in\R^m$, and a set $\L\subset\R^n$, it would be interesting to
study under which conditions there exists a control input
$u:[0,\infty)\to \R^m$ such that $\L$ is an almost global attractor
set. Finally, we think that statement (ii) of Theorem~\ref{thm:main}
may be replaced with the following: The Perron-Frobenius semigroup on
$L^1(\rline^n)$ converges to a Radon measure supported on $\L$, in the
sense of weak$^*$ convergence, i.e., denoting by $C_0(\R^n)$ the set
of continuous functions with compact support, for any initial density
$\rho\in L^1(\R^n)$ and for any $\varphi\in C_0(\R^n)$, the limit
$\lim_{t\to\infty} \int_{\R^n}\varphi(x)[\tline_t\rho](x)\dd x$ exists
and it defines a Radon measure supported on $\L$.

\section*{Acknowledgement}

This research has been supported by the Israel Science Foundation
(ISF), grant no. 2802/21.


\end{document}